\documentclass{amsart}

\usepackage{amsmath} 
\usepackage{amssymb}

\newcount\skewfactor
\def\mathunderaccent#1#2 {\let\theaccent#1\skewfactor#2
\mathpalette\putaccentunder}
\def\putaccentunder#1#2{\oalign{$#1#2$\crcr\hidewidth
\vbox to.2ex{\hbox{$#1\skew\skewfactor\theaccent{}$}\vss}\hidewidth}}
\def\name{\mathunderaccent\tilde-3 }

\newcommand{\forces}{\Vdash} 

\newcommand{\rest}{{\restriction}}
\newcommand{\dom}{{\rm dom}}

\newcommand{\cA}{{\mathcal A}}

\newcommand{\cF}{{\mathcal F}}

\newcommand{\cK}{{\mathcal K}}

\newcommand{\bbP}{{\mathbb P}}
\newcommand{\nbQ}{{\name{\bbQ}}}
\newcommand{\bbQ}{{\mathbb Q}}

\newcommand{\bV}{{\mathbf V}}

\newcommand{\cZ}{{\mathcal Z}}

\newtheorem{theorem}{Theorem}[section]

\newtheorem{proposition}[theorem]{Proposition}

\theoremstyle{definition}

\newtheorem{definition}[theorem]{Definition}
\newtheorem{hypothesis}[theorem]{Hypothesis}

\theoremstyle{remark}

\begin{document}

\title[Explicit example]{Explicit example of collapsing $\kappa^+$ in
  iteration of $\kappa$--proper forcings}

\author{Andrzej Ros{\l}anowski}
\address{Department of Mathematics\\
University of Nebraska at Omaha\\
Omaha, NE 68182-0243, USA}
\email{aroslanowski@unomaha.edu}

\subjclass{Primary 03E40; Secondary: 03E05}
\date{August 03, 2018}

\begin{abstract}
We give an example of iteration of length $\omega$ of
$(<\kappa)$--complete $\kappa^+$--cc forcing notions with the limit
collapsing $\kappa^+$. The construction is decoded from the proof of
Shelah \cite[Appendix, Theorem 3.6(1)]{Sh:f}.
\end{abstract}

\maketitle 

\section{Introduction}
Since 1980s it has been known that there is no straightforward
generalization of properness to the context of iterations with
uncountable supports.  The cannonical reason for that situation was
the failure of club uniformization for colorings on ladder systems
given by Shelah in \cite[Appendix]{Sh:b}, \cite[Appendix, Theorem
3.6(2)]{Sh:f}. 

In a series of papers \cite{RoSh:655, RoSh:860, RoSh:777, RoSh:888,
  RoSh:890, RoSh:942, RoSh:1001}, Ros{\l}anowski and Shelah 
presented several properties of $({<}\kappa)$--strategically
  complete forcing notions implying that their $\kappa$--support
  iterations do not collapse $\kappa^+$. Those properties were
  carefully crafted to ``cover'' nice forcing notions without
  colliding with the bad example of a uniformization forcing. The need
  for the carefull work was typically justified by saying that
  {\em some iteration of uniformization forcings must fail properness\/} 
  without actually spelling out any detailed example. 

Martin Goldstern \cite{mg2017} asked me if I know a simple example of
an $\omega$--step iteration of $(<\kappa)$--complete $\kappa^+$--cc
forcing notions with the limit collapsing $\kappa^+$. My answer then 
was that there cannot be any very simple example, becasue of the works
mentioned above. I was not correct. The purpose of this note is to
give such an explicit, relatively simple,  example. The argument given here
is actually included in some form in the proof of Shelah \cite[Appendix,
Theorem 3.6(1)]{Sh:f}. But the advantage of writing it down explicitly is
that, unlike  \cite[Appendix, Theorem 3.6(2)]{Sh:f}, the argument
applies to inaccessible $\kappa$ as well. So, while for innaccessible
$\kappa$ the theory of $\kappa$--support iterations appear to be much richer
(and easier), we still have to work hard to ensure properness of the limit.   
\medskip

\noindent{\bf Notation}:\qquad Our notation is rather standard and
compatible with that of classical textbooks (like Jech
\cite{J}). However, in forcing we keep the older convention that {\em
  a stronger condition is the larger one}. 

\begin{enumerate}
\item Ordinal numbers will be denoted be the lower case initial letters of
  the Greek alphabet $\alpha,\beta,\gamma,\delta$. Finite ordinals
  (non-negative integers) will be denoted by $k,n$.
\item The letter $\kappa$ will denote a regular uncountable cardinal such
  that $\kappa=\kappa^{<\kappa}$.
\item For a set $\cA$, the family of all subsets of $\cA$ of size $\kappa$
  is denoted $[\cA]^\kappa$ and the family of all sequences of length
  $<\kappa$ with values in $\cA$ is called ${}^{\kappa>}\cA$.
\item For a forcing notion $\bbP$, all $\bbP$--names for objects in
  the extension via $\bbP$ will be denoted with a tilde below (e.g.,
  $\name{\tau}$, $\name{X}$), and $\name{G}_\bbP$ will stand for the
  canonical $\bbP$--name for the generic filter in $\bbP$.
\end{enumerate}

\section{The example}

For the rest of this note we keep the following assumptions.

\begin{hypothesis}
\label{bashyp}  
We assume that:
\begin{enumerate}
\item $\kappa$ is an uncountable cardinal satisfying $\kappa^{<\kappa}
  =\kappa$, and 
\item $\cF\subseteq {}^\kappa\kappa$ is a family of size $\kappa^+$. 
\end{enumerate}
\end{hypothesis}

\begin{definition}
  \label{forcing}
  \begin{enumerate}
\item An {\em $\cF$--coloring\/} is a sequence $\bar{H}=\langle
  H_f:f\in\cF\rangle$  such that $H_f:\kappa\longrightarrow
  {}^{\kappa>} \kappa$ (for $f\in\cF$).  
\item For an $\cF$--coloring $\bar{H}$ we define a forcing notion
  $\bbQ(\bar{H})$ as follows.  

\noindent{\bf A condition in $\bbQ(\bar{H})$}\quad is a tuple $p=(\gamma^p,
e^p, v^p,u^p,h^p)$ such that 
\begin{enumerate}
\item[(a)] $\gamma^p<\kappa$, $e^p\subseteq \gamma^p+1$ is a closed set with
  $\max(e^p)= \gamma^p$,  
\item[(b)]  $v^p\in [\cF]^{<\kappa}$, $u^p=\{f\rest \alpha: \alpha\in e^p\ 
  \wedge \ f\in v^p\}$, 
\item[(c)] if $f,g\in v$ are distinct, then $f\rest\gamma^p\neq
  g\rest\gamma^p$,   
\item[(d)] $h^p:u^p\longrightarrow  {}^{\kappa>} \kappa$.
\end{enumerate}
\noindent{\bf The order $\leq$ of $\bbQ(\bar{H})$}\quad is such that $p\leq 
q$ if and only if 
\begin{enumerate}
\item[(i)] $\gamma^p\leq \gamma^q$ and $e^p=e^q\cap (\gamma^p+1)$, and  
\item[(ii)] $v^p\subseteq v^q$ (so also $u^p\subseteq u^q$) and $h^p\subseteq
  h^q$, and 
\item[(iii)] if $f\in v^p$ and $\alpha\in e^q\setminus e^p$, then $h^q(f\rest
  \alpha) =H_f(\alpha)$. 
\end{enumerate}
\end{enumerate}
\end{definition}

\begin{proposition}
\label{basicobs}
Assume $\kappa=\kappa^{<\kappa}$ is an uncountable cardinal and let
$\bar{H}$ be an $\cF$--coloring.
\begin{enumerate}
\item $\big(\bbQ(\bar{H}),\leq\big)$ is indeed a partial order. 
\item The forcing notion $\bbQ(\bar{H})$ is $({<}\kappa)$--complete and  
$|\bbQ(\bar{H})|=\kappa^+$. 
\item The forcing notion $\bbQ(\bar{H})$ satisfies the $\kappa^+$--chain  
condition (as a matter of fact, it has the $\kappa^+$--Knaster property), so
it is also $\kappa$--proper (in the standard sense).  
\end{enumerate}
\end{proposition}

\begin{proof}
(2)\quad To show the completeness of $\bbQ(\bar{H})$, suppose that
$\bar{p}=\langle  p_\alpha:\alpha< \delta\rangle\subseteq
\bbQ(\bar{H})$ is increasing, $\delta<\kappa$ is limit. 

If $\gamma^{p_\alpha}$ are eventually constant, say
$\gamma^{p_\alpha}=\gamma^{p_{\alpha^*}}$ for $\alpha^*\leq\alpha<\delta$,
then also $e^{p_\alpha}=e^{p_{\alpha^*}}$ for
$\alpha^*\leq\alpha<\delta$. Set $\gamma=\gamma^{p_{\alpha^*}}$, $e=
e^{p_{\alpha^*}}$, $v=\bigcup\{v^{p_{\alpha}}:\alpha<\delta\}$,
$u=\bigcup\{u^{p_{\alpha}}:\alpha<\delta\}$ and
$h=\bigcup\{h^{p_{\alpha}}:\alpha<\delta\}$. Easily, $(\gamma,e,v,u,h)\in
\bbQ(\bar{H})$ is an upper bound to $\bar{p}$. 

Otherwise we may assume without loss of generality that the sequence
$\langle \gamma^{p_\alpha}:\alpha<\delta\rangle$ is strictly increasing. Let
$\gamma=\sup\big(\gamma^{p_\alpha}:\alpha<\delta\big)$, 
$e=\bigcup\big\{e^{p_\alpha}:\alpha<\delta\big\}\cup\{\gamma\}$,
$v=\bigcup\big\{v^{p_\alpha}:\alpha<\delta\big\}$,
$u=\bigcup\big\{u^{p_\alpha}:\alpha<\delta\big\}\cup\{f\rest\gamma:f\in v\}$
and note that these objects satisfy demands \ref{forcing}(2)(a--c). Thus, in
particular, for distinct $f,g\in v$ we have $f\rest \gamma\neq
g\rest\gamma$. Hence may define $h:u\longrightarrow {}^{\kappa>}\kappa$ so
that $h^{p_\alpha}\subseteq h$ for all $\alpha<\delta$ and
$h(f\rest\gamma)=H_f(\gamma)$ for $f\in v$.  Easily, $(\gamma,e,v,u,h)\in
\bbQ(\bar{H})$ is an upper bound to $\bar{p}$. 
\medskip

\noindent (3)\quad Suppose that $\langle p_\alpha:\alpha <\kappa^+
\rangle$ is a sequence of distinct conditions from
$\bbQ(\bar{H})$. Since $\kappa^{<\kappa}=\kappa$ we may pick 
$\cA\subseteq \kappa^+$ of size $\kappa^+$ and $\gamma,e,u, h$ such that for
each $\alpha\in \cA$: 
\[\gamma^{p_\alpha}=\gamma,\quad e^{p_\alpha}=e,\quad u^{p_\alpha}=u
\quad\mbox{  and }\quad h^{p_\alpha}=h.\]
Suppose now $\alpha<\beta$ are from $\cA$. Set $v_0=v^{p_\alpha}\cup
v^{p_\beta}$ and choose $\gamma_0$ such that $\gamma<\gamma_0<\kappa$ and 
\[(\forall f,g\in v_0)(f\neq g\ \Rightarrow f\rest \gamma_0\neq g\rest
\gamma_0).\]
Then put $e_0=e\cup\{\gamma_0\}$ and $u_0=u\cup \{f\rest \gamma_0: f\in
v\}$. Plainly, we may now define $h_0:u_0\longrightarrow {}^{\kappa>}\kappa$
so that $h^{p_\alpha}= h^{p_\beta}\subseteq h_0$ and $h_0(f\rest\gamma_0)= 
H_f(\gamma_0)$ for $f\in v_0$.   

Easily, $(\gamma,e,v,u,h)\in \bbQ(\bar{H})$ is a condition stronger than
both $p_\beta$ and $p_\alpha$.
\end{proof}

\begin{definition}
\label{namedef}
Let $\bar{H}$ be an $\cF$--coloring. We define $\bbQ(\bar{H})$--names
$\name{E}$ and $\name{h}$ by
\[\forces_{\bbQ(\bar{H})}\mbox{`` } \name{E}=\bigcup\big\{e^p:p\in
  \name{G}_{\bbQ(\bar{H})}\big\}\ \wedge\ \name{h}=\bigcup\big\{h^p:p\in 
  \name{G}_{\bbQ(\bar{H})}\big\}\mbox{ ''}.\]
\end{definition}

\begin{proposition}
\label{nameprop}
\begin{enumerate}
\item $\forces_{\bbQ(\bar{H})}\mbox{`` } \name{E}\subseteq \kappa$ is a club
  of $\kappa$ ''. 
\item $\forces_{\bbQ(\bar{H})}$`` $\name{h}$ is a function with domain
  $\bigcup\big\{u^p:p\in \name{G}_{\bbQ(\bar{H})}\big\}$ and values in
  ${}^{\kappa>}\kappa$ and such that for each $f\in\cF$, for some
  $\alpha^f<\kappa$, if $\alpha\in \name{E}\setminus \alpha^f$, then $f\rest 
  \alpha\in \dom(\name{h})$ and $\name{h}(f\rest\alpha)=H_f(\alpha)$ ''. 

\end{enumerate}
\end{proposition}

\begin{proof}
First note that for every $f\in \cF$ and $\alpha<\kappa$ the set 
\[\cZ_\alpha^f= \big\{p\in\bbQ(\bar{H}): f\in v^p\ \wedge\
  \alpha<\gamma^p\big\}\]
is open dense in $\bbQ(\bar{H})$. Then the assertions easily follow by
the definition of the forcing (particularly, by the definition of the order). 
\end{proof}

\begin{theorem}
\label{main}
Assume $\kappa,\cF$ are as in Hypothesis \ref{bashyp}. Then there is an
iteration $\langle\bbP_n,\nbQ_n:n<\omega\rangle$ such that  
  \begin{enumerate}
\item for each $n<\omega$, for some $\bbP_n$--name $\name{\bar{H}}_n$ we 
  have
\[\forces_{\bbP_n}\mbox{`` }\name{\bar{H}}_n\mbox{ is an $\cF$--coloring
  and }\nbQ_n=\bbQ(\name{\bar{H}}_n)\mbox{ '',}\]
so also 
\[\forces_{\bbP_n}\mbox{`` $\nbQ_n$ satisfies $\kappa^+$--cc '',}\]
\item the full support limit $\bbP_\omega$ collapses $\kappa^+$ while
  preserving all cardinalities and cofinalities $\leq\kappa$.
  \end{enumerate}
\end{theorem}

\begin{proof}
We define $\nbQ_n$ and $\name{\bar{H}}_n$ by induction on $n<\omega$.

We start with letting $H^0_f(\alpha)=\langle\rangle$ (for $f\in\cF$ and
$\alpha<\kappa$). Then put $\bar{H}_0=\langle H^0_f: f\in \cF\rangle$ and
$\bbQ_0=\bbQ(\bar{H}_0)$.  

Suppose we have defined $\bbP_k,\nbQ_k, \name{\bar{H}}_k$ for $k<n$ and 
$\forces_{\bbP_k} \nbQ_k=\bbQ(\name{\bar{H}}_k)$. Let $\bbP_n=
\bbP_{n-1}*\nbQ_{n-1}$. It should be clear that the forcing notion
$\bbP_n$ preserves cofinalities and cardinalities and forces
$\kappa^{<\kappa}=\kappa$. By our choice of $\nbQ_k$ for each $k<n$ and we
may choose $\bbP_n$--name $\name{E}^k$ such that  
\[\forces_{\bbP_n}\mbox{`` } \name{E}^k\subseteq \kappa\mbox{ is the 
    club added by the $k$-th coordinate forcing (cf \ref{namedef},
    \ref{nameprop}) ''.}\] 
Now suppose that $f\in \cF$. Let $\name{H}^n_f$ be a $\bbP_n$--name for a
function from $\kappa$ to ${}^{\kappa>}\kappa$ such that   
\[\forces_{\bbP_n}  \name{H}^n_f(\alpha)=f\rest\min\big( \bigcap_{k<n}
  \name{E}^k\setminus (\alpha+1)\big)\quad\mbox{ for each
  }\alpha<\kappa.\]  
(Remember, $\name{E}^k$ are names for clubs of $\kappa$). Then
we let $\name{\bar{H}}_n=\langle \name{H}^n_f:f\in\cF \rangle$ and
$\nbQ_n=\bbQ(\name{\bar{H}}_n)$.  
\medskip

Now, the limit $\bbP_\omega$ is $({<}\kappa)$--complete, so it
preserves cardinalities and cofinalities $\leq \kappa$. Also
$\forces_{\bbP_\omega} \kappa^{<\kappa}=\kappa$. We will argue 
that $\bbP_\omega$ collapses $\kappa^+$ by showing that
$\forces_{\bbP_\omega} |\cF|=\kappa$.

For $n<\omega$ let $\name{E}^n$ be a $\bbP_{n+1}$--name for a club of
$\kappa$ as above and $\name{h}_n$ be a $\bbP_{n+1}$--name for a a partial
function from ${}^{\kappa>}\kappa$ to ${}^{\kappa>}\kappa$ added on $n$-th
coordinate (cf \ref{namedef}, \ref{nameprop}).   

Assume $G\subseteq \bbP_\omega$ is generic over $\bV$ and let us work
in $\bV[G]$. 

Suppose towards contradiction that $|\cF|=\kappa^+$. For each $n<\omega$,
the set $\big(\name{E}^n\big)^G$ is a club of $\kappa$, and therefore the
intersection $E=\bigcap\limits_{n<\omega}\big(\name{E}^n\big)^G$ is also a  
club of $\kappa$.  By \ref{nameprop}(2), for every $f\in\cF$ and $n\in
\omega$ we may fix $\alpha^f_n\in E$ such that 
\begin{enumerate}
\item[$(\heartsuit)$] $\displaystyle\Big(\forall \alpha\in
  (\name{E}^n)^G\setminus \alpha^f_n\Big) \Big(f\rest\alpha\in
  \dom\big((\name{h}_n)^G\big) \ \wedge\ (\name{h}_n)^G(f\rest\alpha)=
  (\name{H}_f^n)^G (\alpha)\Big)$,
\end{enumerate}
and we let $\alpha^f_*=\sup(\alpha^f_n:n<\omega)\in E$. Since
$|\cF|=\kappa^+$, we may find {\em distinct\/} $f,g\in\cF$ such that 
\[\alpha^f_*=\alpha_*^g=\alpha^*\mbox{ and }\quad f\rest \alpha^*=g\rest
\alpha^*.\]  
The set $C\stackrel{\rm def}{=}\{\alpha\in E:f\rest\alpha=g\rest\alpha\}$ is
nonempty (as $\alpha^*\in C$), closed and bounded. Let $\alpha^+=\max(C)$
and for $n<\omega$ let $\beta_n=\min\Big(\bigcap\limits_{k<n}
\big(\name{E}^k\big)^G\setminus (\alpha^++1)\Big)$. By the definition of
$\name{H}^n_f,\name{H}^n_g$ and $(\heartsuit)$ we know that for each $n$:
\[f\rest\beta_n= \big(\name{H}^n_f\big)^G(\alpha^+)= \big(\name{h}_n\big)^G
  (f\rest\alpha^+) = \big(\name{h}_n\big)^G (g\rest\alpha^+) = 
\big(\name{H}^n_g\big)^G(\alpha^+)=g\rest\beta_n.\] 
Let $\beta=\sup(\beta_n:n<\omega)$. Then $\beta\in E$, $\beta>\alpha^+$ and
$f\rest \beta=g\rest \beta$, a contradiction. 
\end{proof}

Assuming $\kappa,\cF$ are as in Hypothesis \ref{bashyp}, we may use
the proof of Theorem \ref{main} to argue that the following statement
is {\bf not} true:
\begin{quotation}
{\em for every $\cF$--coloring $\bar{H}$ there are a club $E\subseteq
  \kappa$ and a function $h: {}^{\kappa>}\kappa
  \longrightarrow {}^{\kappa>}\kappa $ such that 
\[\big(\forall f\in \cF\big)\big(\exists \alpha_f<\kappa\big)\big(
  \forall \alpha\in E\setminus
  \alpha_f\big)\big(h(f\rest\alpha)=H_f(\alpha)\big).\] 
}
\end{quotation}
Consequently, if $\cK$ is a class of forcing notions including all
forcings of the form $\bbQ(\bar{H})$ (for an $\cF$--coloring
$\bar{H}$), then the Forcing Axiom ${\mathbf{FA}}_\kappa(\cK)$
fails. This is yet another example of limitations on possible
extensions of Martin Axiom.


\begin{thebibliography}{10}

\bibitem{mg2017}
Martin Goldstern.
\newblock E-mail message to A.~Ros{\l}anowski, July 4, 2017. 

\bibitem{J}
Thomas Jech.
\newblock {\em {Set theory}}.
\newblock Springer Monographs in Mathematics. Springer-Verlag, Berlin, 2003.
\newblock The third millennium edition, revised and expanded.

\bibitem{RoSh:655}
Andrzej Roslanowski and Saharon Shelah.
\newblock {Iteration of $\lambda$-complete forcing notions not collapsing
  $\lambda^+$.}
\newblock {\em International Journal of Mathematics and Mathematical Sciences},
  28:63--82, 2001.
\newblock arxiv:math.LO/9906024.

\bibitem{RoSh:860}
Andrzej Roslanowski and Saharon Shelah.
\newblock {Reasonably complete forcing notions}.
\newblock {\em Quaderni di Matematica}, 17:195--239, 2006.
\newblock arxiv:math.LO/0508272.

\bibitem{RoSh:777}
Andrzej Roslanowski and Saharon Shelah.
\newblock {Sheva-Sheva-Sheva: Large Creatures}.
\newblock {\em Israel Journal of Mathematics}, 159:109--174, 2007.
\newblock arxiv:math.LO/0210205.

\bibitem{RoSh:888}
Andrzej Roslanowski and Saharon Shelah.
\newblock {Lords of the iteration}.
\newblock In {\em Set Theory and Its Applications}, volume 533 of {\em
  Contemporary Mathematics (CONM)}, pages 287--330. Amer. Math. Soc., 2011.
\newblock arxiv:math.LO/0611131.

\bibitem{RoSh:890}
Andrzej Roslanowski and Saharon Shelah.
\newblock {Reasonable ultrafilters, again}.
\newblock {\em Notre Dame Journal of Formal Logic}, 52:113--147, 2011.
\newblock arxiv:math.LO/0605067.

\bibitem{RoSh:942}
Andrzej Roslanowski and Saharon Shelah.
\newblock {More about $\lambda$--support iterations of $({<}\lambda)$--complete
  forcing notions}.
\newblock {\em Archive for Mathematical Logic}, 52:603--629, 2013.
\newblock arxiv:1105.6049.

\bibitem{RoSh:1001}
Andrzej Roslanowski and Saharon Shelah. 
\newblock {The last forcing standing with diamonds}. 
\newblock {\em Fundamenta Mathematicae}, accepted. 
\newblock arxiv:1406.4217. 

\bibitem{Sh:b}
Saharon Shelah.
\newblock {\em {Proper forcing}}, volume 940 of {\em {Lecture Notes in
  Mathematics}}.
\newblock {Springer-Verlag, Berlin-New York, xxix+496 pp}, 1982.

\bibitem{Sh:f}
Saharon Shelah.
\newblock {\em {Proper and improper forcing}}.
\newblock {Perspectives in Mathematical Logic}. {Springer}, 1998.

\end{thebibliography}

\end{document}